\documentclass{article}

\usepackage{amssymb, latexsym,pdfsync,amsmath,amsthm, ulem,hyperref}
\usepackage{fullpage}
\newtheorem{theorem}{Theorem}

\newtheorem{lemma}{Lemma}

\theoremstyle{definition}

\newtheorem{remark}{Remark}

\newcommand{\FF}{\mathbb{F}}
\newcommand{\Fp}{\mathbb{F}_p}
\newcommand{\Fq}{\mathbb{F}_q}
\newcommand{\Fqn}{\mathbb{F}_{q^n}}

\newcommand{\HH}{\mathcal H}

\newcommand{\D}{\mathcal D}
\newcommand{\G}{\mathcal G}
\newcommand{\C}{\mathcal C}

\def\F{\mathbb{F}}
\def\K{\mathbb{K}}
\def\Fq{{\mathbb{F}}_q}

\def\Aut{\mathrm{Aut}}

\def\PG{\mathrm{PG}}

\def\GL{\mathrm{GL}}
\def\PGammaL{\mathrm{P\Gamma L}}
\def\GammaL{\mathrm{\Gamma L}}
\def\dim{\mathrm{dim}}

\def\Tr{\mathrm{Tr}}
\def\End{\mathrm{End}}
\def\tr{\mathrm{tr}}

\newcommand{\npmatrix}[1]{\left( \begin{matrix} #1 \end{matrix} \right)}

\newcommand{\rank}{\mathrm{rank}}

\begin{document}
\title{A new family of linear maximum rank distance codes}
\author{John Sheekey}
\date{\today}
\maketitle

\begin{abstract}
In this article we construct a new family of linear maximum rank distance (MRD) codes for all parameters. This family contains the only known family for general parameters, the Gabidulin codes, and contains codes inequivalent to the Gabidulin codes. This family also contains the well-known family of semifields known as Generalised Twisted Fields. We also calculate the automorphism group of these codes, including the automorphism group of the Gabidulin codes.
\end{abstract}

\section{Preliminaries}

\subsection{Rank metric codes}

Delsarte introduced rank metric codes in \cite{Delsarte1978}. A {\it rank metric code} $\C$ is a subset of a matrix space $M = M_{m\times n}(\F)$, $m\leq n$, $\F$ a field, equipped with the distance function $d(X,Y) := \rank(X-Y)$. A rank metric code is called $\K${\it -linear} if it forms an $\K$-subspace of $M$ for some subfield $\K \leq \F$.  A {\it maximum rank distance code} (MRD-code) is a rank metric code over a finite field $\Fq$ meeting the Singleton-like bound $|\C| \leq q^{n(m-d+1)}$, where $d$ is the minimum distance of $\C$. If $\C$ is an $\FF_{q_0}$-linear MRD code in $M_{m \times n}(\Fq)$ with $d=m-k+1$ for a subfield $\F_{q_0}$ of $\Fq$, we say that $\C$ has parameters $[nm,nk,m-k+1]_{q_0}$, with the subscript omitted when there is no ambiguity. If a code is closed under addition (which coincides with the definition of $\Fp$-linear, where $p$ is prime and $q=p^e$), we say it is {\it additive}.

Delsarte \cite{Delsarte1978} and Gabidulin \cite{Gab1985} constructed linear MRD-codes over the finite field $\Fq$ for every $k,m$ and $n$. In the literature these are usually called {\it (generalised) Gabidulin codes}, although the first construction was by Delsarte. When $n=m$ and $k=1$, MRD-codes correspond to algebraic structures called {\it quasifields}, see Subsection \ref{ssec:sem}. Cossidente-Marino-Pavese \cite{CoMaPa} recently constructed non-linear MRD-codes for $n=m=3$, $k=2$.  When $n=m$ and $1<k<n-1$, no other linear MRD-codes were known. In this paper we will construct a new family of linear MRD-codes for each $k$, and we will show that they contain codes inequivalent to the generalised Gabidulin codes.

A good overview of MRD-codes can be found in \cite{Ravagnani}. Similar problems for symmetric, alternating and hermitian matrices have been studied in for example \cite{GabPil2004}, \cite{Schmidt}, \cite{GoQu2009b}, \cite{GoLaShVa2014}.

\subsection{(Pre)semifields and quasifields}
\label{ssec:sem}

A finite {\it presemifield} is a division algebra with a finite number of elements in which multiplication is not necessarily associative; if a multiplicative identity element exists, it is called a {\it semifield}. We refer to \cite{LaPo2011} for background, definitions, and terminology. Presemifields are studied in equivalence classes known as {\it isotopy classes}. Presemifields of order $q^n$ with centre containing $\F_{q_0}\leq \Fq$ and {\it left nucleus} containing $\Fq$ are in one-to-one correspondence with $[n^2,n,n]_{q_0}$ MRD-codes in $M_n(\Fq)$ (i.e. $\F_{q_0}$-linear MRD-codes with $k=1$). In the theory of semifields, such spaces are called {\it semifield spread sets}. Many constructions for finite semifields are known, see for example \cite{Kantor2006} and \cite{LaPo2011}. There are two important operations defined on presemifields; the {\it dual}, which is the opposite algebra, and the {\it transpose}. These together form a chain of six (isotopy classes of) semifields, known as the {\it Knuth orbit}.

A {\it quasifield} is an algebraic structure satisfying the axioms of a division algebra, except perhaps left distributivity. Quasifields are in one-to-one correspondence with MRD-codes with $k=1$ which are not necessarily linear (see \cite{Dembowski}). Explicit statements of the correspondence between semifields, quasifields and MRD-codes can be found in \cite{DeKiWaWi}.

\subsection{Equivalence}
\label{ssec:equiv}

There are different concepts of equivalence for rank metric codes, see for example \cite{HuaWan}, \cite{Berger}, \cite{Morrison}. In this paper, two rank metric codes $\C,\C' \subset M_n(\Fq)$ will be said to be {\it equivalent} if there exist invertible $\Fq$-linear transformations $A,B$ and a field automorphism $\rho \in \Aut(\Fq)$ such that $\C' = A\C^{\rho} B := \{AX^{\rho}B:X \in \C\}$ where $X^{\rho}$ is the matrix obtained from $X$ by applying $\rho$ to each entry. These are precisely the linear isometries for the rank-metric, as in \cite{HuaWan}. Note that each of these operations preserve the rank distance, and they form a group. The subgroup fixing $\C$ will be called the {\it automorphism group} of $\C$, and is denoted by $\Aut(\C)$.

We call the set $\hat{\C} := \{\hat{X}:X \in \C\}$ the {\it adjoint} of $\C$, where $\hat{X}$ denotes the adjoint of $X$ with respect to some non-degenerate symmetric bilinear form. This form is often chosen so that the adjoint is precisely matrix transposition, though we will not assume this. Note that taking the adjoint also preserves rank distance, and is often included in the definition of equivalence. However we find it more convenient to omit it, and if $\C$ is equivalent to $\hat{\C'}$, we say that $\C$ and $\C'$ are {\it adjoint-equivalent}.

When $k=1$ and $\C_1$, $\C_2$ are linear, equivalence corresponds precisely to the presemifields associated to each code $\C_i$ being {\it isotopic}, while adjoint-equivalence corresponds to one presemifield being isotopic to the transpose of the other. 

%
%We borrow this terminology, and may refer to codes being isotopic with the obvious meaning. The {\it autotopy group} of a code is the group $\{(A,B) \in \GL(n,q)^2:A\C B = \C\}$. 

\subsection{Subspace codes}

A {\it subspace code} is a set of subspaces of a finite vector space, with the distance function $d_s(U,V) = \dim(U)+\dim(V)-2\dim(U \cap V)$. If all elements of the code have the same dimension, it is called a {\it constant dimension code}. These codes were introduced by Koetter and Kschischang \cite{KoKs2008}, and have applications in random network coding. Rank metric codes define constant dimension subspace codes in the following way (see for example \cite{GaYa2010}). Given an $m\times n$ matrix $X$ we define the subspace $S_X = \{(u,Xu): u \in \Fq^n\}$ of $\Fq^{n+m}$. Clearly, each $S_X$ is $n$-dimensional, and $d_s(S_X,S_Y) = 2\rank(X-Y)$. Hence an MRD-code with minimum distance $d$ defines a subspace code with minimum distance $2d$.  This is known as a {\it lifted} MRD-code \cite{SiKsKo2008}. However, not every subspace code defines an MRD-code when $d<n$, see for example \cite{HoKiKu}, \cite{LiHo2014}. Many of the best known constructions are constructed by perturbing a lifted Gabidulin code \cite{AiHoLi2016}. In the case $d=n$, we have the well-known correspondence between {\it spreads} and quasifields, and in the linear case between {\it semifield spreads} and semifields.

%Let $\SSS$ be any presemifield, and define $N(C,\SSS)$ to be the number of isotopes of $\SSS$ contained in $C$. Then the following is clear by definition.
%\begin{proposition}
%Suppose $C_1$ and $C_2$ are equivalent MRD-codes, and $\SSS$ any presemifield. Then $N(C_1,\SSS) = N(C_2,\SSS)$.
%\end{proposition}

\subsection{Delsarte's duality theorem}
\label{subsec:dual}

Define the symmetric bilinear form $b$ on $M_{m,n}(\FF)$ by
\[
b(X,Y) := \tr(\Tr(XY^T)),
\]
where $\Tr$ denotes the matrix trace, and $\tr$ denotes the absolute trace from $\Fq$ to $\FF_p$, where $p$ is prime and $q=p^e$. Define the {\it Delsarte dual} $\C^{\perp}$ of an $\Fp$-linear code $\C$ by
\[
\C^{\perp} := \{Y:Y \in M_{m,n}(\Fq), b(X,Y)=0 ~\forall X \in \C\}.
\]

We choose the name Delsarte dual to distinguish from the notion of dual in semifield theory. Delsarte \cite[Theorem 5.5]{Delsarte1978} proved the following theorem, using the theory of association schemes. An elementary proof can be found in \cite{Ravagnani}.
\begin{theorem}{\rm \cite[Theorem 5.5]{Delsarte1978}}
Suppose $\C$ is an $[nm,nk,m-k+1]_p$ MRD code in $M_{m\times n}(\Fq)$. Then the Delsarte dual $\C^{\perp}$ is an $[nm,n(m-k),k+1]_p$ MRD code in $M_{m\times n}(\Fq)$.
\end{theorem}

Note that when $n=m=2$, $k=1$, both $\C$ and $\C^{\perp}$ are semifield spread sets (and correspond to {\it rank two semifields}, see \cite{LaPo2011}). In the context of semifields, this operation is known as the {\it translation dual}, see \cite{LuMaPoTr2008}, and is a special case of the {\it switching} operation defined in \cite{BEL2007}.

It is clear that two codes $\C$ and $\C'$ are equivalent if and only if $\C^{\perp}$ and $\C'^{\perp}$ are equivalent; this was shown in \cite{LuMaPoTr2008} for semifield spread sets, the same proof works for this more general statement. Hence the classification of $[n^2,n^2-n,2]_{q_0}$-codes is equivalent to the classification of semifields of order $q^n$ with nucleus containing $\Fq$ and centre containing $\FF_{q_0}$ up to isotopy.

\subsection{Known classifications and computational results}

The only known classification results for MRD codes are those that follow from classifications of semifields. Semifields of order $q^3$ with centre containing $\Fq$ have been fully classified by Menichetti \cite{Menichetti1977}; they are the spread sets corresponding to either a field or a {\it generalised twisted field}. Because of Section \ref{subsec:dual}, we have a full classification of all $\Fq$-linear MRD codes in $M_3(\Fq)$: they are the spread sets corresponding to either a field or a generalised twisted field (minimum distance $3$), or the Delsarte dual of one of these (minimum distance $2$). Precise conditions for the equivalence of spread sets arising from generalised twisted fields can be found in \cite{BiJhJo1999}. As an example, there are precisely two classes of semifields of order $27$, and so two equivalence classes of codes in $M_3(\FF_3)$ with parameters $[9,3,3]$, and two with parameters $[9,6,2]$.

Further computational classifications of semifields of order $q^n$, and hence MRD codes with parameters $[n^2,n,n]$ and $[n^2,n^2-n,2]$, have been performed for small values of $q$ and $n$; namely $q^n \in \{2^4,2^5,2^6,3^4,3^5,5^4,7^4\}$. See \cite[Table 4]{RuCoRa2012} for a referenced up-to-date summary. Explicit matrix representation for some of these can be found at \cite{DempData}.

%\section{Assorted computational results}

%When $q=3$, $n=3$, there exist precisely two linear MRD-codes with $k=2$ which contain a presemifield, namely $\G_2$ and $\HH_2$. This was ascertained by finding all MRD-codes containing $\Fqn$ or the twisted field of order 27. As these are the only semifields of this order, the result follows.

When $n=4$, $q=2$, there are three isotopy classes of semifields of order $16$, and so three equivalence classes of MRD-codes in $M_4(\FF_2)$ with minimum distance $4$ (and with minimum distance $2$, by duality). It remains to classify those with minimum distance $3$, i.e. $8$ dimensional subspaces  of $M_4(\FF_2)$ where all nonzero elements have rank at least $3$. A computation with the computer algebra package MAGMA shows that there is only one MRD-code with $k=2$ containing a semifield spread set; the Gabidulin code, an $8$-dimensional code of minimum distance $3$. The only semifield spread sets it contains are all isotopic to $\FF_{16}$. Of the other two semifield spread sets, one is maximal as a rank metric code of minimum distance $3$; that is, if we extend it by any other matrix, the resulting $5$-dimensional code will always contain an element of rank at most $2$. The spread set corresponding to the final semifield is contained in a $5$-dimensional code of minimum distance $3$, but not any $6$-dimensional code of minimum distance $3$.

However, this does not complete the classification of such MRD-codes. In \cite{DuGoMcSh2010} it was shown that the set of elements of minimum rank in an MRD-code is partitioned into constant rank subspaces of dimension $n$, each lying in the annihilator of a subspace of dimension $k-1$. Hence an MRD-code with $k=2$ must contain an $n$-dimensional constant rank $n-1$ subspace of $\mathrm{Ann}(u)$ for each $u$. However it is not clear whether such a code must contain an $n$-dimensional constant rank $n$ subspace, i.e. a semifield spread set. This remains an open problem.
 
%Hence we can quickly classify linear MRD-codes with $q=2,3$, $n=3$, and $k=2$. First we calculate all constant rank $2$ subspaces of $\mathrm{Ann}(u)$ up to equivalence for some chosen $u$, then extend them to MRD-codes and calculate the equivalence classes again. When $q=2$, there is only the Gabidulin code $\G_2$. When $q=3$, there are precisely two such codes, namely $\G_2$ and $\HH_2$. This in fact follows from Section \ref{subsec:dual}.

%The non-linear example of Cossidente-Marino-Pavese ($n=3$, minimum distance $2$) contains subspaces isotopic to $\Fqn$. It also has the same weight enumerator as the linear MRD-codes.

%All calculations were carried out using the .

\section{Linearized polynomials, and properties of the Gabidulin code}

Let us consider an $n$-dimensional vector space over $\Fq$, which we will denote by $V(n,q)$. We will often identify $V(n,q)$ with the elements of the field extension $\Fqn$. It is well known that every $\Fq$-linear transformation from $\Fqn$ to itself may be represented by a unique {\it linearized polynomial} of $q$-degree at most $n-1$: that is,
\[
M_n(\Fq) \simeq L_n := \{ f_0 x + f_1 x^q + \ldots + f_{n-1}x^{q^{n-1}} : f_i \in \Fqn\}.
\]
Recall that the $q$-degree of a non-zero polynomial is the maximum $i$ such that $f_i \ne 0$. These linearized polynomials form a ring isomorphic to $M_n(\Fq)$, with the multiplication being composition modulo $x^{q^n}-x$ (which we will denote by $\circ$). The foundations of this theory can be found in \cite{Ore1933}.

It is straightforward to see that linearized polynomial of $q$-degree $k$ has rank at least $n-k$. This follows from the fact that such a polynomial can have at most $q^k$ roots, and hence its kernel (when viewed as a linear transformation) has dimension at most $k$, implying that its rank is at least $n-k$. In fact, this turns out to be a special case of the following more general non-trivial result \cite[Theorem 5]{GoQu2009b}.

\begin{theorem}[\cite{GoQu2009b}]
\label{thm:galpol}
Let $\mathbb{L}$ be a cyclic Galois extension of a field $\FF$ of degree $n$, and suppose that $\sigma$ generates the Galois
group of $\mathbb{L}$ over $\FF$. Let $k$ be an integer satisfying $1\leq k< n$, and let $f_0,f_1,\ldots f_{k-1}$ be elements of $\mathbb{L}$, not all zero. Then the $\FF$-linear transformation defined as
\[
f(x) =  f_0x+f_1x^{\sigma}+\cdots+f_{n-1}x^{\sigma^{k-1}}
\]
has rank at least $n-k$.
\end{theorem}

Taking $\mathbb{L}=\Fqn$, $\FF = \Fq$, and  $x^{\sigma} = x^q$ returns the above statement about linearized polynomials. Furthermore, if we take $x^{\sigma} = x^{q^s}$ for some $s$ relatively prime to $n$, then we get that a linearized polynomial of the form 
\[
f_0 x + f_1 x^{q^s} + \ldots + f_{k-1}x^{q^{s(k-1)}}
\]
has rank at least $n-k$. Thus letting $\G_{k,s}$ denote the set of linearized polynomials of this form, for a fixed $s$ and $k$, will give us  an MRD-code. 
\[
\G_{k,s} := \{ f_0 x + f_1 x^{q^s} + \ldots + f_{k-1}x^{q^{s(k-1)}} : a_i \in \Fqn\}.
\]

These are the {\it generalised Gabidulin codes} \cite{GaKs2005}, and are MRD-codes with dimension $nk$ and minimum rank-distance $n-k+1$. 

We define $\G_k = \G_{k,1}$, which is then the set of linearized polynomials of degree at most $k-1$, i.e.
\[
\G_k := \{ f_0 x + f_1 x^q + \ldots + f_{k-1}x^{q^{k-1}} : a_i \in \Fqn\}.
\]
These were first constructed by Delsarte \cite{Delsarte1978}, though in much of the literature they are referred to as Gabidulin codes. For much of the remainder of this paper, we will restrict ourselves to considering the family $\G_k$, though analogous results hold for $\G_{k,s}$. Each code $\G_{1,s}$ is a semifield spread set, and all are equivalent and correspond to the field $\Fqn$.

\begin{remark}
It should be noted that the map 
\[
f_0x+f_1x^{q}+\cdots+f_{n-1}x^{q^{k-1}}\mapsto f_0 x + f_1 x^{q^s} + \ldots + f_{k-1}x^{q^{s(k-1)}}
\]
does {\it not} preserve the rank distance. It was shown in \cite{GaKs2005} that there exist codes in $\G_{k,s}$ inequivalent to any in $\G_k$ for particular values of $k,s$ and $q$. The question of equivalence between generalised Gabidulin codes will be further addressed in Remark \ref{rem:Hgal}.
\end{remark}

\begin{remark}
\label{rem:asvec}
Rank metric codes are sometimes viewed as codes in $(\Fqn)^m$. The theories are basically identical (see e.g. \cite{Morrison}), and the correspondence is as follows. If we choose $m$ elements $e_0,e_1,\ldots,e_{m-1}$ of $\Fqn$, linearly independent over $\Fq$ and spanning a subspace $U$, then we identify the linearized polynomial $f$ with the $m$-tuple
\[
v_f := (f(e_o),f(e_1),\ldots,f(e_{m-1}))\in (\Fqn)^m.
\]
Then the corresponding weight function on $(\Fqn)^m$ is given by $w(v_0,\ldots,v_{m-1}) = \dim_{\Fq} \langle v_0,\ldots,v_{m-1}\rangle$. It is straightforward to check that the weight of $v_f$ is then equal to the rank of the restriction of $f$ to $U$.

The main difference between the two settings is that a code in $(\Fqn)^m$ with this weight function are called {\it linear} if it is an $\Fqn$-subspace, which corresponds to the set of linearized polynomials forming an $\Fqn$-subspace of $L_n$. Clearly the generalised Gabidulin codes are all $\Fqn$-linear. Note however that $\Fqn$-linearity is not preserved by the equivalence as defined in this paper.
\end{remark}

The actions of $\GL(n,q)\times \GL(n,q)$ and $\Aut(\Fq)$ on $M_n(\Fq)$ (as defined in Section \ref{ssec:equiv})  can be translated to an actions on $L_n$ as follows. Given a pair of linearized polynomials $(g,h)\in L_n\times L_n$, where both $g$ and $h$ are invertible as linear transformations on $\Fqn$ (i.e. have no nontrivial roots in $\Fqn$), we define a map from $L_n$ to itself by
\[
f^{(g,h)} := g \circ f \circ h \mod x^{q^n}-x.
\]

Given a linearized polynomial $f$ and an automorphism $\rho$ of $\Fq$, we define $f^{\rho}(x) = f(x^{\rho^{-1}})^{\rho}\mod x^{q^n}-x$. Note that $f^{\rho}$ can be obtained by simply applying $\rho$ to each of the coefficients of $f$. If $q=p^e$ for $p$ a prime, and $x^{\rho} = x^{p^i}$, then $f^{\rho} = x^{p^i}\circ f \circ x^{p^{ne-i}}$. Hence, extending the above notation in a natural way, any automorphism of a code $\C$ can be written as $(g \circ x^{p^i},x^{p^{ne-i}} \circ h)$ for some linearized polynomials $g,h$.

The set $S := \{\alpha x: \alpha \in \Fqn^{\times}\}$ is a subgroup of $\GL(n,q)$ isomorphic to $\Fqn^{\times}$, and is what is known as a {\it Singer cycle}. Then $\G_{k,s}$ is fixed under the actions of $S$ defined by $f \mapsto (\alpha x) \circ f$ and  $f \mapsto f \circ (\alpha x)$, for $\alpha \in \Fqn^{\times}$. It is also fixed by $f \mapsto x^p \circ f \circ x^{p^{ne-1}}$, and hence by the group 
\[
\{(\alpha x^{p^i},\beta x^{p^{ne-i}}):\alpha,\beta \in \Fqn^{\times}, i \in \{0,\ldots,ne-1\}\}.
\]
We will show later that this is in fact the full stabiliser of $\G_k$ for each $k$.

The {\it adjoint} of a linearized polynomial $a = \sum_{i=0}^{n-1} a_i x^{q^i}$ with respect to the symmetric bilinear form $(a,b) \mapsto \Tr(ab)$ (where $\Tr$ denotes the absolute trace from $\Fqn$ to $\Fp$) is given by $\hat{a} = \sum_{i=0}^{n-1} a_{n-i}^{q^i} x^{q^i}$. It is easy to check that $x^{q^k} \circ \hat{\G}_k = \G_k$, and hence we have the following.

\begin{lemma}
Each Gabidulin code $\G_k$ is equivalent to its adjoint $\hat{\G_k}$.
\end{lemma}

We define the symmetric bilinear form $b$ on linearized polynomials  by
\[
b\left(\sum_{i=0}^{n-1} f_i x^{q^i},\sum_{i=0}^{n-1} g_i x^{q^i}\right) = \Tr\left(\sum_{i=0}^{n-1} f_i g_i\right).
\] 
We may choose an $\Fq$-basis for $\Fqn$ in such a way that this bilinear form coincides with the form introduced in Section \ref{subsec:dual}, which motivates the use of the same symbol $b$. Note that $\sum_{i=0}^{n-1} f_i g_i$ is the coefficient of $x$ in $f\hat{g}$. The following Lemma is immediate.
\begin{lemma}
The Delsarte dual $\G_k^{\perp}$ of a Gabidulin code $\G_k$ is equivalent to $\G_{n-k-1}$.
\end{lemma}

Note also that the Gabidulin codes form a chain:
\[
\G_1 \leq \G_2 \leq \cdots \leq \G_{n-1} \simeq M_n(\Fq).
\]
We now consider which subspaces of a Gabidulin code $\G_k$ are equivalent to another Gabidulin code $\G_r$.

\begin{theorem}
\label{thm:gabembed}
A subspace $U$ of $\G_k$, $k\leq n-1$, is equivalent to $\G_r$ if and only if there exist invertible linearized polynomials $f,g$ such that  $U = \G_r ^{(f,g)} =  \{f \circ a \circ g : a \in \G_r\}$, where $f_0 =1$, and $\deg_q(f) + \deg_q(g) \leq k-r$. 
\end{theorem}

\begin{proof}
Clearly if $f$ and $g$ are invertible linearized polynomials satisfying the condition on degrees, then $U$ is contained in $\G_k$, and isotopic to $\G_r$.

Note that for any $0 \ne \beta \in \Fqn$ and any $j \in \{0,\ldots,n-1\}$, we have that $\{f \circ a \circ g : a \in \G_r\} = \{f \circ (\beta x^{q^j}) \circ a \circ (\beta^{-1} x)^{q^{n-j}}\circ g : a \in \G_r\}$, and hence we may assume without loss of generality that $f_0 =1$.

Consider $f \circ \alpha x^{q^j} \circ g$, where $\alpha \in \Fqn$. Then the coefficient of $x^{q^m}$ is
\[
a_{m,j}(\alpha) := \sum_{i=0}^{n-1} f_i g_{m-i-j}^{q^i} \alpha^{q^i},
\]
where indices are taken modulo $n$. If $U$ is contained in $\G_k$, we must have that for each $m \geq k$, $j \leq r-1$, $a_{m,j}(\alpha)$ is zero for every $\alpha\in \Fqn$. Hence for all $m \geq k$, $j \leq r-1$ and $i \in \{0,\ldots,n-1\}$ we have that
\[
f_i g_{m-i-j} = 0.
\]
As $f_0 \ne 0$, we get that $g_m = 0$ for all $m \geq k$. Let $\deg_q(f) = s$, $\deg_q(g) = t$, and so $f_s g_t \ne 0$. Then $g_{m-s-r+1} = 0$ for all $m \in \{k,\ldots,n-1\} \ne \emptyset$, and hence $t \le k-s-r$, proving the claim. 
\end{proof}

In \cite[Proposition 6]{Morrison} it was shown that the group $\{(\alpha x,\beta x):\alpha,\beta \in \Fqn^{\times}\}$ is a subgroup of the automorphism group of $\G_k$. Note that the result in \cite[Theorem 4]{Morrison} refers to a different definition of equivalence to the definition used in this paper. We now give a complete description of the automorphism group of the Gabidulin codes.

\begin{theorem}
\label{thm:gabaut}
The automorphism group of the Gabidulin code $\G_k$ is given by
\[
\{(\alpha x^{p^i},\beta x^{p^{ne-i}}):\alpha,\beta \in \Fqn^{\times}, i \in \{0,\ldots,n-1\}\}.
\]
\end{theorem}

\begin{proof}
Clearly the given group is a subgroup of the automorphism group of $\G_k$. Suppose $(\G_k^{\rho})^{(f, g)} = \G_k$ for some invertible linearized polynomials $f,g$ and some $\rho \in \Aut(\Fq)$.  As $\G_k^{\rho} = \G_k$ for all $\rho\in\Aut(\Fq)$, we may assume that $\rho$ is the identity. Then $f = f' \circ (\alpha x^{q^i})$ for some $\alpha \in \Fqn^{\times}$, $i \in \{0,\ldots,n-1\}$, where $f'$ such that $f'_0 = 1$. Let $g' = x^{q^i} \circ g$. Then $ \G_k^{(f', g')} = \G_k$, and by the proof of Theorem \ref{thm:gabembed}, we must have $\deg_q(f')+\deg_q(g') = 0$. Hence $f'=x$ and $g' = \beta^{q^i} x$ for some $\beta \in \Fqn^{\times}$, and so $(f,g) = (\alpha x^{q^i},\beta x^{q^{n-i}})$, proving the claim.
\end{proof}

\begin{remark}
An analogous proof shows that the automorphism group of any generalised Gabidulin code $\G_{k,s}$ is equal to the automorphism group of $\G_k$.
\end{remark}

%\section{Duals of MRD-code}
%
%Define a symmetric bilinear form on $M_n(\Fq)$ by....
%
%Define a symmetric bilinear form on $L_n$ by $b(f,g) = \Tr((f\circ g)_0)$. Delsarte \cite[Theorem 5.5]{Delsarte1978} showed the following.
%
%\begin{theorem}
%Suppose $\C$ is an $\FF_{q_0}$-linear MRD-code with parameters $[n^2,nk,n-k+1]$. Then $\C^{\perp}$ is a linear MRD-code with parameters $[n^2,n(n-k),k+1]$. Moreover, $\C$ is equivalent to $\C'$ if and only if $\C^{\perp}$ is equivalent to $\C'^{\perp}$.
%\end{theorem}
%
%\begin{corollary}
%If $\C$ is a semifield spread set, then $\C^{\perp}$ is an $\Fq$-linear $[n^2,n^2-n,2]$-code, and the number of isotopy classes of semifields with centre containing $\Fq$ is equal to the number of isotopy classes of $\Fq$-linear $[n^2,n^2-n,2]$-codes.
%\end{corollary}
%
%An elementary proof can be found in [Ravagnani?]., [Gow]
%
%It is clear that $\G_k^{\perp} \simeq \G_{n-k}$.

\section{Construction of new linear MRD-codes}

The following Lemma is key to our construction. The result follows from \cite[Theorem 10]{GoQu2009a}, and in the case where $q$ is prime from \cite{Ore1933}. We give a proof for completeness. We denote the field norm from $\Fqn$ to $\Fq$ by $N$, i.e. $N(x) = x^{\frac{q^n-1}{q-1}}$.
\begin{lemma}
\label{lem:coeffs}
Suppose $f$ is a linearized polynomial of $q$-degree $k$. If $f$ has rank $n-k$, then $N(f_0) = (-1)^{kn}N(f_k)$.
\end{lemma}

\begin{proof}
For any $k$-dimensional $\Fq$-subspace $U$ of $\Fqn$, there is a unique monic linearized polynomial of $q$-degree $k$ that annihilates $U$ (that is, contains $U$ in its set of roots). We denote this by $m_U$ and call it the minimal polynomial of $U$,. Every linearized polynomial of degree $k$ annihilating $U$ is an $\Fqn$-multiple of $m_U$, and hence it suffices to prove the result for any particular linearized polynomial of degree $k$ annihilating $U$.

Choose an $\Fq$-basis $\{u_0,u_1,\ldots,u_{k-1}\}$ of $U$, and define a linearized polynomial $f$ as the determinant of a $(k+1)\times (k+1)$ matrix as follows:
\[
f(x) := \det\npmatrix{x&x^q&\cdots&x^{q^{k}}\\
u_0&u_0^q&\cdots&u_0^{q^{k}}\\
\vdots&\vdots&\ddots&\vdots\\
u_{k-1}&u_{k-1}^q&\cdots&u_{k-1}^{q^{k}}} = f_0x+f_1x^q+\cdots f_kx^k.
\]
Then it is clear to see that $f$ annihilates $U$, because plugging in any $u\in U$ for $x$, we get that the top row is an $\Fq$-linear combination of the remaining rows. Furthermore, expanding along the top row we see that $f_0=(-1)^kf_k^q$, and so $N(f_0) = (-1)^{kn}N(f_n)$, proving the claim.
%We proceed by induction on $k$. If $k=1$ and $U = \langle \alpha\rangle$, then $U$ is annihilated by $\alpha x^q - \alpha^q x$, and so the case $k=1$ holds. Suppose the case $k-1$ holds. Let $\{\alpha_1,\ldots,\alpha_k\}$ be an $\Fq$-basis for $U$, and let $U' = \langle\alpha_1,\ldots,\alpha_{k-1}\rangle$. Define $g := m_{U'}$, and  $f := (g(\alpha_k) x^q - g(\alpha_k)^qx) \circ g$. Then $f$ annihilates $U$. Now $f_k = g(\alpha_k) g_{k-1}^q$ and $f_0 = -g(\alpha_k)^q g_0$. By the induction hypothesis, $N(g_{k-1}) = (-1)^{(k-1)n}N(g_0)$, and so the result follows.
\end{proof}

Note that the converse is not true for $k>1$; that is, $N(f_0) = (-1)^{kn}N(f_k)$ does not imply that $f$ has rank $n-k$.

\begin{theorem}
Let $\HH_k(\eta,h)$ denote the set of linearized polynomials of $q$-degree at most $k\leq n-1$ satisfying $f_k = \eta f_0^{q^h}$, with $\eta$ such that $N(\eta) \ne (-1)^{nk}$, i.e.
\[
\HH_k(\eta,h) := \{ f_0 x + f_1 x^q + \ldots + f_{k-1}x^{q^{k-1}} + \eta f_0^{q^h}x^{q^k} : f_i \in \Fqn\}.
\]
Then each $\HH_k(\eta,h)$ is an MRD-code with the same parameters as $\G_k$. 
\end{theorem}

\begin{proof}
It is clear that $\HH_k(\eta,h)$ has dimension $nk$ over $\Fq$. As $\deg(f) \leq q^k$ for all $f\in \HH_k(\eta,h)$, we have that $\rank(f) \geq n-k$. By Lemma \ref{lem:coeffs}, $\rank(f) >n-k$, and hence $\rank(f) \geq n-k+1$ for all $f\in \HH_k(\eta,h)$. It follows that $\HH_k(\eta,h)$ is an MRD-code with parameters $[n^2,nk,n-k+1]$, as claimed.
\end{proof}

\begin{theorem}
\label{thm:Hperp}
The adjoint of $\HH_k(\eta,h)$ is equivalent to $\HH_k(\eta^{-q^{k-h}},k-h)$, and the Delsarte dual satisfies $\HH_k(\eta,h)^{\perp} = \HH_{n-k}(-\eta^{q^{n-h}},n-h)$. 
%Moreover, $\HH_k(\eta,h)$ is not equivalent to $\G_k$ unless $k\in \{1,n-1\}$ and $h \in \{0,1\}$.
%
%The automorphism group of $\HH_k(\eta,h)$ is $\{(ax^{q^i},bx^{q^{-i}}):a,b \in \Fqn, ab^{q^k}=(ab)^{qh}\}$.
%, and $\G_k$ and $\HH_k(\eta,h)$ are (more or less) the unique codes with this property.
\end{theorem}

\begin{proof}
This follows from a straightforward calculation. 
%
%
%Suppose $\G_k$ is equivalent to $\HH_k(\eta,h)$, which is contained in $\G_{k+1}$. By Theorem \ref{thm:gabembed}, if $k < n-1$ then we must have $\HH_k(\eta,h) = f \G_k g$, for some $f,g$ with $f$ monic and $f_0 \ne 0$, and $\deg_q(f)+\deg_q(g) \leq 1$. If $\deg_q(f) =1$ then $f = x^q-\alpha x$ for some $\alpha$ with $N(\alpha) \ne 1$, and $\deg_q(g) =0$. Hence the leading coefficient of an element $f \circ a \circ g$ is $a_{k-1}^q$, while the coefficient of $x$ is $-\alpha a_0$, and so we must have $a_{k-1}^q = -\eta \alpha^{q^h} a_0^{q^h}$ for all $a_0,a_{k-1} \in \Fqn$. This is possible if and only if $k=h=1$ and $-\eta \alpha^{q}=1$. If $\deg_q(f) =0$, we get that $g= x^q - \beta x$, and a similar calculation shows that $k=1$, $h=0$ and $-\eta \beta = 1$, proving the result.
%
%NOTE: Automorphism group proves that $\HH$ is inequivalent to the generalised gabidulin codes... add proof here
\end{proof}

Note that $\HH_k(0,h) = \G_k$, so this family includes the Gabidulin codes. We now prove that there are new MRD-codes in this family. First we require the following lemma.
%
%\begin{theorem}
%Suppose $k\notin\{1,n-1\}$. Then $\HH_k(\eta,h)$ is isomorphic to $\HH_k(\nu,j)$ if and only if...
%
%
%Furthermore, the isomorphism group of $\HH_k(\eta,h)$ is given by
%\[
%\{(\alpha x^{p^i},\beta x^{p^{-i}}):\alpha,\beta \in \Fqn, \alpha^{1-q^h} (\beta^{q^k-q^h})^{p^i} \eta^{p^i}=\eta\}
%\]
%Hence $\HH_k(\eta,h)$ is not equivalent to $\G_{k,s}$ unless $k\in \{2,n-1\}$ and $h \in \{0,1\}$. Furthermore, $\HH_k(\eta,h)$ is equivalent to $\HH_k(\nu,j)$ if and only if $j=h$ and there exist $\alpha,\beta \in \Fqn$ such that $\nu = \alpha^{1-q^h} (\beta^{q^k-q^h})^{p^i} \eta^{p^i}$.
% \end{theorem}
%\begin{proof}
%Note that both $\HH_k(\eta,h)$  and $\HH_k(\nu,j)$ are contained in $\G_{k+1}$, and each contain the subspace $\G_{k-1} \circ x^q$. Suppose $\HH_k(\eta,h) = f \HH_k(\nu,j) g$. Then $f \G_{k-1} (x^q \circ g) \leq \G_{k+1}$. If $r$ is the smallest integer such that $f_r \ne 0$, then setting $f' = f \circ x^{q^{-r}}$ and $g' = x^{q^{r+1}} \circ g$, we get that $f' \G_{k-1} g' \leq \G_{k+1}$, with $f_0 \ne 0$. By Theorem \ref{thm:gabembed}, we must have $\deg(f')+\deg(g') \leq 2$.
%
%If $\deg(f')=s$, then $\deg(g') \leq 2-s$. Let $b = b_0 + \eta b_0^{q^h} x^{q^k}+\sum_{i=1}^{k-1} b_i x^{q^i}$ be a generic element of $\HH_k(\eta,h)$. Then the coefficient of $x$ in $fbg$ is $b_0 f_0 g_0$, while the coefficient of $x^{q^k}$ is $a_{k-2}^{q^s}f_s g_{2-s}^{q^{s}}$. 
%
%
%
% 
%  Then by Theorem \ref{thm:gabembed}, $U = f\G_{k-1}g$ for some $f,g$ with $f_0 \ne 0$ and $\deg(f)+\deg(g) \leq 2$. 
%
%\end{proof}
%

\begin{lemma}
\label{lem:uniquesubgab}
Suppose $\eta \ne 0$. Then there is a unique subspace of $\HH_k(\eta,h)$ equivalent to $\G_{k-1}$, unless $k\in\{1,n-1\}$, or $k=2$ and $h \leq 2$.
\end{lemma}

\begin{proof}
Note that $\HH_k(\eta,h)$ is contained in $\G_{k+1}$. Suppose $U$ is a subspace of $\HH_k(\eta,h)$ equivalent to $\G_{k-1}$. Then by Theorem \ref{thm:gabembed}, $U = \G_{k-1}^{(f,g)}$ for some $f,g$ with $f_0 \ne 0$ and $\deg(f)+\deg(g) \leq 2$. 

If $\deg(f)=s$, then $\deg(g) \leq 2-s$. Let $b = \sum_{i=0}^{k-2} b_i x^{q^i}$ be a generic element of $\G_{k-1}$. Then the coefficient of $x$ in $f\circ b\circ g$ is $b_0 f_0 g_0$, while the coefficient of $x^{q^k}$ is $a_{k-2}^{q^s}f_s g_{2-s}^{q^{s}}$. Hence we must have $\eta(b_0 f_0 g_0)^{q^h} = b_{k-2}^{q^s}f_s g_{2-s}^{q^{s}}$ for all $b_0,b_{k-2}\in \Fqn$. As $\eta\ne 0$ and $f_0, f_s \ne 0$, if $k>2$ this is possible if and only if $g_0 =g_{2-s}=0$. Hence $s=0$, $g=g_1x^q$, implying $U = \G_{k-1}\circ x^q$, proving the claim.

If $k=2$, we get $\eta(b_0 f_0 g_0)^{q^h} = b_0^{q^s}f_s g_{2-s}^{q^{s}}$ for all $b_0\in \Fqn$, which is possible only if $h=s\leq 2$.
\end{proof}

This lemma allows us to calculate the automorphism group, and hence prove that the family $\HH_k(\eta,h)$ contains codes inequivalent to any generalised Gabidulin code, and therefore contains new MRD codes.

\begin{theorem}
Suppose $k\notin\{1,n-1\}$, $\eta \ne 0$. Then the automorphism group of $\HH_k(\eta,h)$ is
\[
\{(\alpha x^{p^i},\beta x^{p^{-i}}):\alpha,\beta \in \Fqn, \alpha^{1-q^h} (\beta^{q^k-q^h})^{p^i} \eta^{p^i}=\eta\}
\]
Hence $\HH_k(\eta,h)$ is not equivalent to $\G_{k,s}$ unless $k\in \{1,n-1\}$ and $h \in \{0,1\}$. Furthermore, $\HH_k(\eta,h)$ is equivalent to $\HH_k(\nu,j)$ if and only if $j=h$ and there exist $\alpha,\beta \in \Fqn$ such that $\nu = \alpha^{1-q^h} (\beta^{q^k-q^h})^{p^i} \eta^{p^i}$.
 \end{theorem}

\begin{proof}
Suppose first that $k\ne 2$. By Lemma \ref{lem:uniquesubgab}, there is a unique subspace of $\HH_k(\eta,h)$ equivalent to $\G_{k-1}$. Hence any isomorphism from $\HH_k(\eta,h)$ to $\HH_k(\nu,j)$ is also an automorphism of $\G_{k-1}$. By Theorem \ref{thm:gabaut}, these are all of the form $(\alpha x^{p^i},\beta x^{p^{ne-i}})$ for some $\alpha,\beta \in \Fqn^{\times}, i \in \{0,\ldots,ne-1\}$.

Now the coefficient of $x$ in a generic element of $(\alpha x^{p^i}) \circ \HH_k(\eta,h)\circ (\beta x^{p^{ne-i}})$ is equal to $\alpha \beta^{p^i} f_0^{p^i}$, while the coefficient of $x^{q^k}$ is $\alpha \beta^{p^{ke+i}} (\eta f_0^{q^h})^{p^i}$. Hence for this to lie in $\HH_k(\nu,j)$ we must have $\alpha \beta^{p^{ke+i}} (\eta f_0^{q^h})^{p^i} = \nu (\alpha \beta^{p^i} f_0^{p^i})^{q^j}$ for all $f_0 \in \Fqn$. This occurs if and only if $j=h$ and $\alpha \beta^{p^{ke+i}} \eta^{p^i} = \nu (\alpha \beta^{p^i})^{q^h}$, proving the last claim. Setting $\eta=\nu$ gives the automorphism group of $\HH_k(\eta,h)$.

 If $k=2<n-1$, then the result follows by taking into account Theorem \ref{thm:Hperp}, and noting that $\Aut(\C^{\perp}) = \{(\hat{f},\hat{g}):(f,g) \in \Aut(\C)\}$.

It is clear that $\Aut(\HH_k(\eta,h))$ is strictly smaller than $\Aut(\G_{s,k})$, unless $\eta=0$, proving that $\HH_k(\eta,h)$ is not equivalent to $\G_{k,s}$ for $k\notin\{1,n-1\}$.

When $k=1$, the code $\HH_k(\eta,h)$ is a semifield spread set corresponding to a {\it generalised twisted field}, as introduced by Albert \cite{Albert1961}, i.e a presemifield with multiplication $x\circ y = xy + \eta x^{q^h}y^q$, where $N(-\eta) \ne 1$. The equivalence and automorphisms of these follow from \cite{BiJhJo1999}. By duality, we get the result for $k=n-1$, completing the proof.
\end{proof}

Note that $\Aut(\HH_k(\eta,h))$ contains the Singer cycle $\{(\alpha x,x):\alpha \in \Fqn\}$ if $h=0$. In fact, $\HH_k(\eta,h)$ is $\Fqn$-linear in this case, showing that there exist $\Fqn$-linear MRD-codes which are not equivalent to generalised Gabidulin codes. This was also proved by computer calculation in \cite{MaTrAlcoma} for $n=4$, $k=2$, $q=3$. When $h=k$, $\Aut(\HH_k(\eta,h))$ contains the Singer cycle $\{(x,\beta x):\beta \in \Fqn\}$. However this code is not $\Fqn$-linear.

\begin{remark}
As noted in the proof above, when $k=1$, the code $\HH_k(\eta,h)$ is a semifield spread set corresponding to a {\it generalised twisted field}, see \cite{Albert1961}. Indeed, it was the construction of these twisted fields that provided the inspiration for this new family. For this reason we propose to name this family of codes {\it twisted Gabidulin codes}. 
\end{remark}

\begin{remark}
Note that when we view $\HH_k(\eta,h)$ as a code in $(\Fqn)^n$, it is $\Fqn$-linear if and only if $h=0$. 
\end{remark}

\begin{remark}
The special case $\HH_2(\eta,1)$ was discovered independently in \cite{OtOzAlcoma}. 
\end{remark}

\begin{remark}
Let $\phi_1,\phi_2$ be linearized polynomials. Define $\HH_k(\phi_1,\phi_2)$ to be the set of linearized polynomials of degree at most $k$ with $f_0 = \phi_1(a)$, $f_k = \phi_2(a)$; i.e
\[
\HH_k(\phi_1,\phi_2) = \{\phi_1(a)x+f_1x^q+\cdots+f_{k-1}x^{q^{k-1}}+\phi_2(a)x^{q^k}:a,f_1,\ldots,f_{k-1}\in \Fqn\}.
\] 
Then $\HH_k(\phi_1,\phi_2)$ is an MRD-code with parameters $[n^2,nk,n-k+1]$ if and only if 
\[
N(\phi_1(x)) \ne (-1)^{kn} N(\phi_2(x))
\]
for all $x \in \Fqn^{\times}$. This is equivalent to finding an $(n-1)$-dimensional subspace in $\PG(2n-1,q)$ disjoint from the projective hypersurface $\mathcal{Q}_{n-1,q}$ induced by the set
\[
\{(x,y) : x,y \in \Fqn, N(x) = N(y)\}.
\]
This hypersurface was studied in detail in \cite{LaShZa2013}. The only known pairs of functions $(\phi_1,\phi_2)$ satisfying this condition are equivalent to the pair $(x,\eta x^{q^h})$, and $\HH(x,\eta x^{q^h}) = \HH(\eta,h)$ by definition. Note that for each pair $(\phi_1,\phi_2)$, the code $\HH_k(\phi_1,\phi_2)$ is contained in $G_{k+1}$ and contains $G_{k-1}\circ x^q$, and every MRD-code satisfying this property is of the form $\HH_k(\phi_1,\phi_2)$ for some $\phi_1,\phi_2$. Note also that every such $(\phi_1,\phi_2)$ defines a presemifield, with multiplication
\[
x\phi_1(y) + x^{q^h}\phi_2(y).
\]
\end{remark}

\begin{remark}
\label{rem:Hgal}
Suppose $\mathbb{L}$ is a cyclic Galois extension of a field $\F$, and $\sigma$ a generator of the Galois group $\mathrm{Gal}(\mathbb{K}:\F)$. As mentioned in Theorem \ref{thm:galpol}, an analogous result regarding the rank of endomorphisms of the form $f(x) = \sum_{i=0}^{n-1} f_i x^{\sigma^i}$, $f_i \in \mathbb{L}$, holds. Thus MRD-codes exist in $M_n(\F)$ over any field $\F$ which admits a cyclic Galois extension of degree $n$. When $\mathbb{L}=\Fqn$, $\F=\Fq$, $x^{\sigma} = x^{q^s}$, this coincides precisely with the generalised Gabidulin codes. These codes in characteristic zero were studied in \cite{AuLoRo2013}, and have applications to {\it space-time coding}.

Because of \cite[Theorem 10]{GoQu2009a}, the generalisation of Lemma \ref{lem:coeffs}, we can similarly define MRD-codes $\HH_k(\eta,h;\sigma)$ as
\[
\HH_k(\eta,h;\sigma) = \{f_0 x + f_1 x^{\sigma} + \ldots + f_{k-1}x^{\sigma^{k-1}} + \eta f_0^{\sigma^h}x^{\sigma^k} : f_i \in \mathbb{L}\}\subset \End_{\FF}(\mathbb{L}),
\]
where $\eta\in \mathbb{L}$ satisfies $\eta^{1+\sigma+\cdots+\sigma^{n-1}}\ne 1$.
%It is clear that the analogue of Lemma \ref{lem:coeffs} holds in this case; that is, the rank of $f(\sigma)$ is equal to $n-\deg(f)$ only if $N(f_k)=(-1)^{nk}N(f_0)$. 
%
%Hence the codes $\HH_k(\eta,h;\sigma)$ defined as the set of endomorphisms $f(\sigma)$ where $\deg(f) \leq k$ and $f_k = \eta f_0^{\sigma^h}$ 

Then $\HH_k(\eta,h;\sigma)$ is also a MRD-code, and an analogous proof shows that is inequivalent to the Gabidulin codes. When $\mathbb{L}=\Fqn$, $\F=\Fq$, $x^{\sigma} = x^{q^s}$, we denote $\HH_k(\eta,h;\sigma):= \HH(\eta,h;s)$, and then $\HH_k(\eta,h)=\HH_k(\eta,h;1)$ by definition. 

Subsequent to the original submission of this paper, the question of equivalence for $\HH(\eta,h;s)$, $s\ne1$ was addressed in \cite{LuTrZh2016}.
\end{remark}

\begin{remark}
Given an MRD-code $\C$ in $M_n(\F)$, one can define a code $\overline{\C}$ in $M_{m \times n}(\F)$ by deleting the last $(n-m)$ rows from each element of $\C$. This turns out to also be an MRD-code, known as a {\it punctured} code. However it does not necessarily hold that if $\C$ and $\C'$ are inequivalent then $\overline{\C}$ and $\overline{\C'}$ are inequivalent. It is not clear whether or not the punctured codes obtained from $\HH_k(\eta,h)$ are inequivalent to those obtained from generalised Gabidulin codes.
\end{remark}

%We show that codes with this property are in one-to-one correspondence with a certain type of semifield.
%
%Suppose that $K$ is a code contained in $G_{k+1}$ and containing $x^q \circ G_{k-1}$. Then $K$ contains an $n$-dimensional space $U$ skew from $G_{k-1}$, i.e contained in the space $\{ax - bx^{q^k}:a,b \in \Fqn\}$. 
%
%We can identify this with an $n$-dimensional space of $V(2n,q)$, which has the property that $N(a)-N(b) \ne 0$ for all $(a,b) \in U$. 
%
%
%
%It was shown NO IT WASN'T in \cite{LaShZa2013} that the only such subspaces are $\{(y,\eta y^{q^j}):y \in \Fqn\}$ for some $\eta$ with $N(\eta) \ne 1$, which correspond to a $\HH_k$, and $\{(y,0):y \in \Fqn\}$ and $\{(0,y):y \in \Fqn\}$, which correspond to (codes equivalent to) $\G_k$.

\section{Representations as matrices and vectors}

In order to transfer between a set of linearized polynomials and a set of matrices, we choose an $\Fq$-basis for $\Fqn$, and find the matrix $A$ corresponding to multiplication by a primitive element $\alpha$, and the matrix $S$ corresponding to the Frobenius automorphism $x \mapsto x^q$, with respect to this basis. Then the set
\[
\{A^i S^j:i,j \in \{0..n-1\}\}
\]
is a basis for $M_n(\Fq)$. For a linearized polynomial $f(x) = \sum_i f_ix^{q^i}$, we define $r_i$ such that $f_i = \alpha^{r_i}$ for all $i$ such that $f_i \ne 0$. Then $f$ corresponds to the matrix
\[
\sum_{f_i\ne 0} A^{r_i}S^i.
\]

Choosing the basis $\{1,\alpha,\ldots,\alpha^{n-1}\}$ will make the matrix $A$ easy to work with (it will be the companion matrix of the minimal polynomial of $\alpha$), while choosing $\alpha$ to be a normal element and the basis $\{\alpha,\alpha^q,\ldots,\alpha^{q^{n-1}}\}$ will make the matrix $S$ easy to work with (it will be a permutation matrix). We will choose the former in this paper.   We will give some example in the case $q=3$, $n=4$. We choose an element $\alpha$ of $\FF_{81}$ satisfying $\alpha^4=\alpha+1$. We will write the coordinates of an element of $\FF_{81}$ with respect to this basis as columns. In this case we have
\[
A = \npmatrix{0&0&0&1\\1&0&0&0\\0&1&0&0\\0&0&1&1},\quad S = \npmatrix{1&0&1&0\\0&0&1&2\\0&0&1&1\\0&1&1&1},
\]
Then for the Gabidulin code $\G_2$ with minimum distance $3$, which is an $8$-dimensional space, we have a basis $\{A^iS^j:i \in \{0,\ldots,3\},j\in \{0,1\}\}$. Explicitly, this is
\[
\npmatrix{1&0&0&0\\0&1&0&0\\0&0&1&0\\0&0&0&1},\quad\npmatrix{0&0&0&1\\1&0&0&0\\0&1&0&0\\0&0&1&1},\quad
\npmatrix{0&0&1&1\\0&0&0&1\\1&0&0&0\\0&1&1&1},\quad\npmatrix{0&1&1&1\\0&0&1&1\\0&0&0&1\\1&1&1&1},
\]
\[
\npmatrix{1&0&1&0\\0&0&1&2\\0&0&1&1\\0&1&1&1},\quad\npmatrix{0&1&1&1\\1&0&1&0\\0&0&1&2\\0&1&2&2},\quad
\npmatrix{0&1&2&2\\0&1&1&1\\1&0&1&0\\0&1&0&1},\quad\npmatrix{0&1&0&1\\0&1&2&2\\0&1&1&1\\1&1&1&1}.
\]
The first four matrices form a basis for $\G_1$, corresponding to the field $\FF_{81}$. As a code in $\FF_{81}^4$, performing the procedure described in Remark \ref{rem:asvec} on the $\FF_{81}$-basis $\{x,x^q\}$ give the generator matrix
\[
\npmatrix{1&\alpha&\alpha^2&\alpha^3\\1&\alpha^3&\alpha^6&\alpha^9}.
\]
For the new twisted Gabidulin code $\HH_2(\alpha,1) = \{ax+bx^3+\alpha a^3x^{9}:a,b \in \FF_{81}\}$, we can take a basis 
\[
\{A^i+A^{3i+1}S^2:i \in \{0,\ldots,3\}\}\cup \{A^iS:i \in \{0,\ldots,3\}\}.
\]
Explicitly we get
\[
\npmatrix{1&1&0&1\\1&1&2&1\\0&2&1&0\\0&2&2&1},\quad\npmatrix{1&1&1&2\\1&1&2&0\\0&0&2&0\\1&2&2&0},\quad
\npmatrix{2&2&1&2\\1&1&0&0\\2&2&1&2\\2&1&1&1},\quad\npmatrix{1&1&2&2\\0&2&2&0\\2&0&0&1\\0&0&1&1},
\]
\[
\npmatrix{1&0&1&0\\0&0&1&2\\0&0&1&1\\0&1&1&1},\quad\npmatrix{0&1&1&1\\1&0&1&0\\0&0&1&2\\0&1&2&2},\quad
\npmatrix{0&1&2&2\\0&1&1&1\\1&0&1&0\\0&1&0&1},\quad\npmatrix{0&1&0&1\\0&1&2&2\\0&1&1&1\\1&1&1&1}.
\]
In this case the first four matrices form a basis for $\HH_1(\alpha,1)$, corresponding to a generalised twisted field.

As a code in $\FF_{81}^4$, $\HH_2(\alpha,1)$ is not $\FF_{81}$-linear, but only $\FF_3$-linear. The codewords are then all $\FF_3$-linear combinations of the rows of the matrix
\[
\npmatrix{
\alpha^{28}&\alpha^{13}&\alpha^{17}&\alpha^{5}\\
\alpha^{5}&\alpha^{61}&\alpha^{54}&2\\
\alpha^{13}&\alpha^{55}&\alpha^{30}&\alpha^{50}\\
\alpha^{65}&\alpha^{}&\alpha^{}&\alpha^{}\\
1&\alpha^{3}&\alpha^{6}&\alpha^{9}\\
\alpha&\alpha^{4}&\alpha^{7}&\alpha^{10}\\
\alpha^{2}&\alpha^{5}&\alpha^{8}&\alpha^{11}\\
\alpha^{3}&\alpha^{6}&\alpha^{9}&\alpha^{12}}.
\]
As a code in $\FF_{81}^4$, the code $\HH_2(\alpha,0)$ is $\FF_{81}$-linear, and inequivalent to $\G_2$. It has generator matrix
\[
\npmatrix{1&\alpha^3&\alpha^7&\alpha^9\\\alpha^{28}&\alpha^{13}&\alpha^{17}&\alpha^5}.
\]

Classification of all $8$-dimensional spaces of $M_4(\FF_3)$ with minimum rank-distance $3$ remains an open problem. Recent work on $\FF_{81}$-linear codes can be found in \cite{MaTrAlcoma}.

\section{MRD-codes, scattered subspaces and scattered linear sets}

The {\it desarguesian spread} $\D$ in the vector space $V(2n,q)$ is the set of $n$-dimensional subspaces obtained by considering $1$-dimensional $\Fqn$-subspaces of $V(2,q^n)$ as $\Fq$-subspaces in $V(2n,q)$. It partitions the nonzero vectors, and any pair of elements of $\D$ intersect trivially. 

A subspace $U\leq V(2n,q)$ is said to be {\it scattered} with respect to $\D$ if $U$ intersects any element of $\D$ in a subspace of dimension at most $1$. The maximum dimension of such a subspace is $n$, and a subspace meeting this bound is called a {\it maximum scattered subspace}. The stabiliser of $\D$ in $\GammaL(2n,q)$ is $\GammaL(2,q^n)$, and so it is natural to consider subspaces of $V(2n,q)$ up to $\GammaL(2,q^n)$-equivalence. These concepts were first introduced in \cite{BlLa00}. Scattered subspaces have interesting connections to many topics, see for example \cite{La2016}.

The elements of $\D$ can be identified with the projective line $\PG(1,q^n)$. A {\it linear set of rank $r$} is a set $L(U) := \{ \langle u \rangle _{\Fqn}:u \in U^{\times}\}\subset \PG(1,q^n)$, where $U$ is an $\Fq$-subspace of $V(2,q^n)$ of dimension $r$. A linear set is said to be {\it scattered} if $|L(U)| = \frac{q^r-1}{q-1}$; this coincides with the subspace $U$ being scattered with respect to $\D$. Two linear sets $L(U)$ and $L(U')$ are said to be {\it equivalent} if there exists an element of $\PGammaL(2,q^n)$ mapping $L(U)$ to $L(U')$. Clearly equivalence of subspaces implies the equivalence of the corresponding linear sets. However, the converse is not true, as will be illustrated at the end of this section.

Given a linearized polynomial $f$, define $U_f := \{(y,f(y)): y \in \Fqn\}\leq V(2n,q)$. Note that every subspace of dimension $n$ is equivalent under $\GL(2,q^n)$ to a space of the form $U_f$. Then $U_f$ is scattered if and only if $\rank(f-\beta x) \geq n-1$ for all $\beta \in \Fqn$. This occurs if and only if $\rank(\alpha f -\beta x) \geq n-1$ for all $\alpha,\beta \in \Fqn$. We will call such a polynomial a {\it scattered polynomial}. Hence the set $\C_f := \{\alpha f - \beta x:\alpha,\beta \in \Fqn\} = \langle x,f\rangle_{\Fqn}$ is an MRD-code of dimension $2n$ and minimum distance $n-1$, i.e. $k=2$. Furthermore, $\C_f$ contains the identity map, and is an $\Fqn$-subspace of $L_n$. Indeed, such codes are in one-to-one correspondence with scattered subspaces of dimension $n$ with respect to $\D$. We show now that the notions of equivalence coincide.

It is straightforward to check that the linear sets $U_f$ and $U_g$ are equivalent if and only if $g = (\alpha f^{\rho}+\beta x)(\gamma f^{\rho} + \delta x)^{-1}$ for some $\alpha,\beta,\gamma,\delta \in \Fqn$, $\rho \in \Aut(\Fq)$, with $\alpha \delta-\beta\gamma\ne 0$ and $(\gamma f^{\rho} + \delta x)$ invertible. Hence if $U_f$ and $U_g$ are equivalent, then $\C_f$ and $\C_g$ are equivalent, as 
\begin{align*}
\C_g &= \langle x,g\rangle_{\Fqn} = \langle x,(\alpha f^{\rho}+\beta x)(\gamma f^{\rho} + \delta x)^{-1}\rangle_{\Fqn} \\
&= \langle \alpha f^{\rho}+\beta x,\gamma f^{\rho} + \delta x\rangle_{\Fqn} (\gamma f^{\rho} + \delta x)^{-1}\\
& = \langle x,f\rangle^{\rho}_{\Fqn} (\gamma f^{\rho} + \delta x)^{-1} = (\C_f)^{\rho} (\gamma f + \delta x)^{-1}.
\end{align*}

For the converse, it can be shown that if $X$ is an invertible linear map, then $X\C_f$ is an $\Fqn$-subspace if and only $X(x)=\alpha x^{\rho}$ for some $\alpha \in \Fqn$ and $\rho \in \Aut(\Fq)$, whence $X\C_f = \C_{f^{\rho}}$. Hence if $\C_g = X\C_f Y$, then $\C_g = \C_{f^{\rho}} Y$. As the identity is in $\C_g$, we get that $Y =  \C_{f^{\rho}}(\gamma f^{\rho} + \delta x)^{-1}$ for some $\gamma,\delta \in \Fqn$, and so $\C_f$ is equivalent to $\C_g$ if and only if $\C_g = (\C_f)^{\rho}(\gamma f^{\rho} + \delta x)^{-1}$. Then the argument from the preceding paragraph can be reversed to show that $U_f$ is equivalent to $U_g$. Hence we have shown the following.
\begin{theorem}
Let $f$ and $g$ be scattered linearized polynomials. Then $U_f$ and $U_g$ are equivalent if and only if $\C_f$ and $\C_g$ are equivalent as MRD-codes.
\end{theorem}

Note that this does not imply that the linear sets $L(U_f)$ and $L(U_g)$ are equivalent if and only if $\C_f$ and $\C_g$ are equivalent; it is possible that two subspaces $U_f$ and $U_g$ define the same linear set, without being equivalent. For example, taking $f(x) =x^{q^s}$, $g(x) = x^{q^t}$, the linear sets $L(U_f)$ and $L(U_g)$ are equal if $\gcd(n,s)=\gcd(n,t)$, while the subspaces $U_f$ and $U_g$, and hence the generalised Gabidulin codes $\C_f = \G_{2,s}$ and $C_g = \G_{2,t}$, are inequivalent unless $s=t$ or $s+t=n$.

%For example, the linear set associated to the Lunardon-Polverino example is equivalent to the Blokhuis-Lavrauw example only when $n=3$. But as we have seen, $G_2$ and $H_2(\eta,h)$ are never equivalent, even when $n=3$.

As far as the author is aware, there are only two known constructions for scattered subspaces of dimension $n$ in $V(2n,q)$. They are those defined by $f(x) = x^{q^s}$, $(n,s)=1$ (Blokhuis-Lavrauw \cite{BlLa00}), and $f(x) = x^q+\eta x^{q^{n-1}}$, with $N(\eta) \ne 1$ (Lunardon-Polverino \cite{LuPo1999}). The first family leads to generalised Gabidulin codes $\G_{2,s}$, while the second lead to codes equivalent to $\HH_2(\eta,1)$. These are the only $\Fqn$-linear codes in the family $\HH_2$, and so we do not obtain any new scattered linear sets from this construction.


\begin{thebibliography}{9}

%\bibitem{Landsberg2012} J.M. Landsberg. Tensors: Geometry and Applications. 2012. Graduate Studies in Mathematics, 128. American Mathematical Society, Providence, RI, 2012. xx+439 pp. ISBN: 978-0-8218-6907-9.
%
%\bibitem{LaSh2014} M. Lavrauw and J. Sheekey. Orbits of the stabiliser group of the Segre variety product of three projective lines. {\em Finite Fields Appl.} 26 (2014) 1--6.




\bibitem{AiHoLi2016}
J. Ai, T. Honold, H. Liu. The expurgation-augmentation method for constructing good plane Subspace Codes, arxiv:1601.01502.

\bibitem{Albert1961} 
A.A. Albert. Generalized twisted fields, {\it Pacific J. Math}, 11 (1961) 1-8.

\bibitem{AuLoRo2013}
D. Augot, P. Loidreau, G. Robert. Rank metric and Gabidulin codes in characteristic zero, {\it Proceedings ISIT 2013}, 509-513.

\bibitem{BEL2007}
S. Ball, G. Ebert, M. Lavrauw. A geometric construction of finite semifields,
{\it J. Algebra} 311 (2007), 117-129.

\bibitem{Berger}
T. Berger. Isometries for rank distance and permutation group of Gabidulin codes, {\it IEEE T. Inform. Theory} 49 (2003) 3016-3019.

\bibitem{BiJhJo1999} 
M. Biliotti, V. Jha, N.L. Johnson. The collineation groups of generalized twisted field planes,
{\it Geom. Dedicata} 76 (1999), 97-126. 

\bibitem{BlLa00}
A. Blokhuis, M. Lavrauw. Scattered spaces with respect to a spread in $\PG(n,q)$,
{\it Geom. Dedicata} 81 (2000), 231-243.

\bibitem{CoMaPa} 
A. Cossidente, G. Marino, F. Pavese. Non-linear maximum rank distance codes, {\it preprint}.

\bibitem{DeKiWaWi}
J. de la Cruz, M. Kiermaier, A. Wassermann, W. Willems. Algebraic structures of MRD Codes, arXiv:1502.02711.

\bibitem{Delsarte1978}
P. Delsarte. Bilinear forms over a finite field, with applications to coding theory, {\it J. Combin. Theory Ser. A} 25 (1978) 226-241.

\bibitem{Dembowski}
P. Dembowski. {\it Finite Geometries}, Springer, 1968.

\bibitem{DempData}
U. Dempwolff. Translation Planes of Small Order. \\
http://www.mathematik.uni-kl.de/$\sim$dempw/dempw\textunderscore Plane.html

\bibitem{DuGoMcSh2010} 
J.-G. Dumas, R. Gow, G. McGuire, J. Sheekey. Subspaces of matrices with special rank properties, {\it Linear Algebra Appl.}  433 (2010), 191-202.

\bibitem{Gab1985}
E. M. Gabidulin. Theory of codes with maximum rank distance, {\it Probl. Inf. Transm.}  21 (1985), 1-12.

\bibitem{GaKs2005}
E. Gabidulin, A. Kshevetskiy. The new construction of rank codes, {\it Proceedings. ISIT 2005}.

\bibitem{GabPil2004}
E. M. Gabidulin, N. I. Pilipchuk. Symmetric rank codes, {\it Probl. Inf. Transm.} 40 (2004), 103-117.

\bibitem{GaYa2010}
M. Gadouleau, Z. Yan. Constant-rank codes and their connection to constant-dimension codes, {\it
IEEE Transactions on Information Theory} 56 (2010), 3207-3216.

\bibitem{GoLaShVa2014}	
R. Gow, M. Lavrauw, J. Sheekey, F. Vanhove. Constant rank-distance sets of hermitian matrices and partial spreads in hermitian polar spaces, {\it Elect. J. Comb.} 21 (2014), P1.26.


\bibitem{GoQu2009a} 
R. Gow, R. Quinlan. Galois theory and linear algebra. {\it Linear Algebra Appl.} 430 (2009), 1778-1789.

\bibitem{GoQu2009b} 
R. Gow, R. Quinlan. Galois extensions and subspaces of alternating bilinear forms with special rank properties. {\it Linear Algebra Appl.} 430 (2009), 2212-2224.



\bibitem{HoKiKu}
T. Honold, M. Kiermaier, S. Kurz. Optimal binary subspace codes of length 6, constant dimension 3 and minimum subspace distance 4, {\it Contemporary Mathematics} 632.

\bibitem{Kantor2006}
W.M. Kantor. {\it Finite semifields}, Finite geometries, groups, and computation, 103-114, Walter de Gruyter GmbH \& Co. KG, Berlin, 2006.

\bibitem{KoKs2008} 
R. Koetter, F.R. Kschischang. Coding for errors and erasures in random network coding. {\it IEEE Trans. on Inf. Theory}, 54 (2008), 3579-3591.

\bibitem{La2016}
M. Lavrauw. Scattered Spaces in Galois Geometry. arXiv:1512.05251v2

\bibitem{LaPo2011}
M. Lavrauw, O. Polverino. {\it Finite semifields}. Chapter in {\it Current research topics in Galois Geometry} (Editors J. De Beule and L. Storme), NOVA Academic Publishers, New York, 2011.

\bibitem{LaShZa2013}
M. Lavrauw, J. Sheekey, C. Zanella.
 On embeddings of minimum dimension of $\PG(n,q)\times \PG(n,q)$,
 {\it Des. Codes Cryptogr.}, 74 (2015) 427-440.
 
\bibitem{LiHo2014}
H. Liu, T. Honold.  A new approach to the main problem of Subspace Coding. arXiv:1408.1181

\bibitem{LuMaPoTr2008}
G. Lunardon, G. Marino, O. Polverino, R. Trombetti. {\it Translation dual
of a semifield}, J. Combin. Theory Ser. A 115 (2008), 1321-1332.

\bibitem{LuPo1999}
G. Lunardon, O. Polverino. Blocking sets and derivable partial spreads, {\it J. Algebr. Comb.} 14 (2001), 49-56.

\bibitem{LuTrZh2016}
G. Lunardon, R. Trombetti, Y. Zhou. Generalized Twisted Gabidulin Codes, arXiv:1507.07855v2.



\bibitem{MaTrAlcoma}
K. Marshall, A-L. Trautmann. Characterizations of MRD and Gabidulin codes. Conference presentation, ALCOMA15. \url{http://user.math.uzh.ch/trautmann/ALCOMA_presentation.pdf}.

\bibitem{Menichetti1977}
G. Menichetti. On a Kaplansky conjecture concerning three-dimensional division algebras over a finite field, 
{\it J. Algebra} 47 (1977), 400-410.

\bibitem{Morrison}
K. Morrison. Equivalence for rank-metric and matrix codes and automorphism groups of Gabidulin codes. {\it IEEE Trans. Inform. Theory} 60 (2014) 7035-7046.

\bibitem{Ore1933}
O. Ore. On a special class of polynomials, {\it Trans. Amer. Math. Soc.} 35 (1933) 559-584. 

\bibitem{OtOzAlcoma}
K. Otal, F. \"Ozbudak. Some non-Gabidulin MRD Codes, Conference presentation, ALCOMA15.

\bibitem{Ravagnani}
A. Ravagnani. Rank-metric codes and their MacWilliams identities, arXiv:1410.1333v2.

\bibitem{RuCoRa2012}
I.F. R\'ua, E. F. Combarro, J. Ranilla. Determination of division algebras with 243 elements, {\it Finite Fields Appl.} 18 (2012) 1148-1155.

\bibitem{Schmidt}
K-U. Schmidt. Symmetric bilinear forms over finite fields with applications to coding theory, arXiv:1410.7184v1

\bibitem{SiKsKo2008}
D. Silva, F. R. Kschischang, R. Koetter. A rank-metric approach to error control in random network coding, {\it IEEE Trans. Inform. Theory}, 54 (2008), 3951-3967.

\bibitem{HuaWan}
Z. X. Wan. {\it Geometry of matrices. In memory of Professor L. K. Hua (1910-
1985)}, World Scientific (1996).

\end{thebibliography}
\end{document}